\documentclass[11pt]{article}
\usepackage{amsthm}
\usepackage{fullpage}
\usepackage{amssymb, amsmath}
\usepackage{hyperref}
\usepackage{graphicx}
\usepackage{titlesec}
\usepackage{caption}
\usepackage{enumitem}
\usepackage{thmtools}
\usepackage{thm-restate}
\newtheorem{thm}{Theorem}
\newtheorem{lem}[thm]{Lemma}

\newtheorem{cor}[thm]{Corollary}
\newtheorem{conj}[thm]{Conjecture}
\newtheorem{obs}[thm]{Observation}
\newtheorem{ques}[thm]{Question}
\newtheorem{claim}[thm]{Claim}

\title{On an $f$-coloring generalization of linear arboricity of multigraphs}
\author{Ronen Wdowinski\thanks{Department of Combinatorics and Optimization, University of Waterloo, Waterloo, ON, Canada. Email: ronen.wdowinski@uwaterloo.ca.}}
\date{}

\begin{document}
\maketitle
\begin{abstract}
Given a multigraph $G$ and function $f : V(G) \rightarrow \mathbb{Z}_{\ge 2}$ on its vertices, a degree-$f$ subgraph of $G$ is a spanning subgraph in which every vertex $v$ has degree at most $f(v)$. The degree-$f$ arboricity $a_f(G)$ of $G$ is the minimum number of colors required to edge-color $G$ into degree-$f$ forests. At least for constant $f$, Truszczy\'nski conjectured that $a_f(G) \le \max \{\Delta_f(G) + 1, a(G)\}$ for every multigraph $G$, where $\Delta_f(G) = \max_{v \in V(G)} \lceil d(v)/f(v) \rceil$ and $a(G)$ is the usual arboricity of $G$. This is a strong generalization of the Linear Arboricity Conjecture due to Akiyama, Exoo, and Harary. In this paper, we disprove Truszczy\'nski's conjecture in a strong sense for general multigraphs. On the other hand, extending known results for linear arboricity, we prove that the conjecture holds for simple graphs with sufficiently large girth, and that it holds for all simple graphs asymptotically. More strongly, we prove these partial results in the setting of directed graphs, where the color classes are required to be analogously defined degree-$f$ branchings.
\end{abstract}

\begin{center}
\end{center}

\section{Introduction}

In this writing, a \textit{multigraph} may have parallel edges but no loops, unless otherwise stated. Given a multigraph $G$ and a function $f : V(G) \rightarrow \mathbb{Z}_{\ge 1}$ on its vertex set, a \textit{degree-$f$ subgraph} (or an \textit{$f$-matching}) is a spanning subgraph $H$ of $G$ such that every vertex $v \in V(G)$ has degree $d_H(v) \le f(v)$ in $H$. If every vertex $v$ has degree exactly $f(v)$, such a subgraph is commonly known as an \textit{$f$-factor}. An \textit{$f$-coloring} of $G$ is an assignment of a color to every edge of $G$ so that each color class is a degree-$f$ subgraph. The \textit{$f$-chromatic index} $\chi_f'(G)$ of $G$ is the minimum number of colors required in an $f$-coloring of $G$. For a vertex subset $S \subseteq V(G)$, we let $e(S)$ denote the number of edges in $G$ with both endpoints in $S$, and we write $f(S) = \sum_{v \in S} f(v)$. 

Hakimi and Kariv \cite{HaKa} introduced the notion of an $f$-coloring as a generalization of the case $f = 1$ of a proper edge-coloring, where $\chi_1'(G)$ is the usual chromatic index of $G$. This paper will study the similar problem of edge-coloring a multigraph into degree-$f$ forests rather than degree-$f$ subgraphs. This problem generalizes the most-studied cases $f = \infty$ (arboricity) and $f = 2$ (linear arboricity) to more general vertex weight functions.

The \textit{arboricity} $a(G)$ of a multigraph $G$ is the minimum number of colors required to edge-color $G$ so that every color class is a forest. A celebrated theorem of Nash-Williams \cite{Na} states that the arboricity of a multigraph $G$ is given by
\begin{align*}
	a(G) = \max_{S \subseteq V(G), |S| \ge 2} \left\lceil \frac{e(S)}{|S|-1} \right\rceil.
\end{align*}
Given a function $f : V(G) \rightarrow \mathbb{Z}_{\ge 2}$, the \textit{degree-$f$ arboricity} $a_f(G)$ of multigraph $G$ is the minimum number of colors required to edge-color $G$ so that every color class is a degree-$f$ forest. A degree-$2$ forest is more commonly known as a \textit{linear forest}, and $a_2(G) = la(G)$ is called the \textit{linear arboricity} of $G$. Unlike arboricity, determining the linear arboricity $la(G)$ of a general multigraph $G$ is NP-hard \cite{Pe}. However, a conjecture known as the Linear Arboricity Conjecture, due to Akiyama, Exoo, and Harary \cite{AkExHa}, asserts that we can always determine the linear arboricity of simple graph to within an additive error of one. Observe that $la(G) \ge \lceil \Delta(G)/2 \rceil$ for every multigraph $G$ because we require at least $\lceil d(v)/2 \rceil$ linear forests to cover the edges incident to vertex $v$. The following was conjectured in \cite{AkExHa}.

\begin{conj}[Linear Arboricity Conjecture] \label{linear-arboricity}
For every simple graph $G$, we have  $la(G) \le \left\lceil (\Delta(G) + 1)/2 \right\rceil$.
\end{conj}

The Linear Arboricity Conjecture has been verified for many classes of simple graphs, including complete bipartite graphs \cite{AkExHa}, series-parallel graphs \cite{Wu2}, planar graphs \cite{Wu1}, and when $\Delta(G) \in \{3, 4, 5, 6, 8, 10\}$ (see \cite{AkExHa, AkExHa2, EnPe, Gu3}). Alon \cite{Al1} proved that the Linear Arboricity Conjecture nearly holds for graphs with sufficiently large girth, and that it holds for all simple graphs asymptotically. Subsequent asymptotic improvements were given by Alon \cite{Al0}, by Ferber, Fox, and Jain \cite{FeFoJa}, and by Lang and Postle \cite{LaPo}, the latter of whom have given the currently best known asymptotic bound $la(G) \le \Delta(G)/2 + O(\Delta(G)^{1/2} (\log \Delta(G))^{1/4})$. The Linear Arboricity Conjecture has also been proven for graphs of bounded sparsity (e.g., bounded degeneracy, treewidth, pseudoarboricity) when the maximum degree is sufficiently large (see \cite{BaBiFrPa, ChHaYu, TaWu, Wd}).

There has not been as much work on the linear arboricity of multigraphs or on the degree-$f$ arboricity $a_f(G)$ for $f \neq 2$. Conjecture \ref{linear-arboricity} does not extend to general multigraphs $G$ since, for example, $la(G) = \Delta(G)$ when $G$ consists of $\Delta(G)$ parallel edges between two vertices. A\"it-djafer \cite{Ai} generalized Conjecture \ref{linear-arboricity} to the statement $la(G) \le \lceil (\Delta(G)+\mu(G))/2 \rceil$ for every multigraph $G$, where $\mu(G)$ is the edge-multiplicity of $G$. She verified this when $\mu(G) \ge \Delta(G)-2$, as well as when $\Delta(G)$ is close to a power of 2 and $\mu(G)$ is close to $\Delta(G)/2$. Caro and Roditty \cite{CaRo} proved that for constant functions $f = t$, every $k$-degenerate simple graph $G$ satisfies $a_t(G) \le \lceil (\Delta(G) + (t - 1)k - 1)/t \rceil$.

This paper focuses on a strong generalization of the Linear Arboricity Conjecture due to Truszczy\'nski  \cite{Tr}. For a multigraph $G$ and function $f : V(G) \rightarrow \mathbb{Z}_{\ge 2}$, define the weighted maximum degree parameter
\begin{align*}
	\Delta_f(G) = \max_{v \in V(G)} \left\lceil \frac{d(v)}{f(v)} \right\rceil.
\end{align*}
Observe that $a_f(G) \ge \Delta_f(G)$ and $a_f(G) \ge a(G)$. Truszczy\'nski conjectured the following when $f$ is a constant function.

\begin{conj} \label{degree-f-arboricity-1}
For every multigraph $G$ and function $f : V(G) \rightarrow \mathbb{Z}_{\ge 2}$, we have
\begin{align*}
	a_f(G) \le \max \{\Delta_f(G) + 1, a(G)\}.
\end{align*}
\end{conj}

More strongly, Truszczy\'nski conjectured that for every multigraph $G$ and integer $t \ge 2$, we have $a_t(G) = \max \{ \lceil \Delta(G)/t \rceil, a(G)\}$ unless $a(G) = \Delta(G)/t$, in which case we have $a_t(G) \in \{ \Delta(G)/t, \Delta(G)/t + 1 \}$. He proved this when $G$ is a complete multigraph with all edges having the same multiplicity, when $G$ is a complete bipartite multigraph with all edges having the same multiplicity, when the underlying simple graph of $G$ is a forest, and when $t \ge \Delta(G) - a(G) + 1$. For the particular case $f = 2$ of linear arboricity, Conjecture \ref{degree-f-arboricity-1} asserts that every multigraph $G$ satisfies
\begin{align*}
	la(G) \le \max \left\{ \left\lceil \Delta(G)/2 \right\rceil + 1, a(G) \right\}.
\end{align*}
On the other hand, Nash-Williams' Theorem \cite{Na} for arboricity above implies that $a(G) \le \lceil (\Delta(G) + 1)/2 \rceil$ for every simple graph $G$. This shows that Conjecture \ref{degree-f-arboricity-1} is close to a generalization of the Linear Arboricity Conjecture, both to multigraphs and to other vertex weight functions $f$.

Our first main result is that Conjecture \ref{degree-f-arboricity-1} is false in a strong sense for general multigraphs $G$, for any fixed constant function $f = t$.

\begin{restatable}{thm}{degreefarboricitylarge} \label{degree-f-arboricity-large}
For every integer $t \ge 2$, there exists a constant $c_t > 1$ such that the following holds. For infinitely many integers $d \ge 2$, there exists a multigraph $G$ such that $\max \{\Delta_t(G), a(G)\} = d$ and $a_t(G) \ge c_t d$. 
\end{restatable}

We will prove Theorem \ref{degree-f-arboricity-large} by exhibiting multigraphs $G_t$ that have large \textit{fractional} degree-$t$ arboricity. We will then obtain the graph $G$ of the theorem by replacing every edge in $G_t$ with many parallel edges, which will make $a_t(G)$ and $\max\{\Delta_t(G), a(G)\}$ grow arbitrarily far apart by a constant factor $c_t > 1$. We will show that we can take the constant $c_t$ in the theorem to satisfy $c_t \ge (4t + 7)/(4t + 6)$ for every integer $t \ge 2$, with slight improvements $c_2 \ge 9/8, c_3 \ge 15/14, c_4 \ge 21/20$ for $t \in \{2,3,4\}$.

The falsity of Conjecture \ref{degree-f-arboricity-1} stands in contrast to other theorems and open conjectures on $f$-colorings. For example, in \cite{Wd} this paper's author proved that a result analogous to Conjecture \ref{degree-f-arboricity-1} holds for the degree-$f$ pseudoarboricity of a multigraph. A \textit{pseudoforest} is a multigraph where every component has at most one cycle (possibly a loop). The \textit{pseudoarboricity} $pa(G)$ of a multigraph $G$ is the minimum number of colors required to edge-color $G$ into pseudoforests. Analogous to Nash-Williams' Theorem, a theorem of Hakimi \cite{Ha1} states that the pseudoarboricity of a multigraph $G$ (possibly with loops) is given by
\begin{align*}
	pa(G) = \max_{S \subseteq V(G), |S| \ge 1} \left\lceil \frac{e(S)}{|S|} \right\rceil.
\end{align*}
Given a multigraph $G$ (possibly with loops) and function $f : V(G) \rightarrow \mathbb{Z}_{\ge 2}$, the \textit{degree-$f$ pseudoarboricity} $pa_f(G)$ of $G$ is the minimum number of colors required to edge-color $G$ into degree-$f$ pseudoforests. It was shown in \cite{Wd} that the degree-$f$ pseudoarboricity has the exact formula
\begin{align*}
	pa_f(G) = \max \{\Delta_f(G), pa(G)\}.
\end{align*}
This result has a similar form to a conjecture on the $f$-chromatic index $\chi_f'(G)$ due to Nakano, Nishizeki, and Saito \cite{NaNiSa}. Their conjecture, in turn, is a generalization of the well-known Goldberg-Seymour Conjecture for the chromatic index \cite{Go1, Se2} (see also \cite{ChJiZa, St}).

On the other hand, it was observed in \cite{Wd} that the above formula for $pa_f(G)$ implies the following approximation of Conjecture \ref{degree-f-arboricity-large}.

\begin{thm}\label{degree-f-arboricity-approximation}
For every multigraph $G$ and function $f : V(G) \rightarrow \mathbb{Z}_{\ge 2}$, we have
\begin{align*}
	a_f(G) \le 2pa_{2f}(G) \le \max \{\Delta_f(G) + 1, 2pa(G)\}.
\end{align*}
\end{thm}

Note that $pa(G) \le a(G) \le 2pa(G)$ for every multigraph $G$. Thus, Theorem \ref{degree-f-arboricity-1} shows that we cannot generally decrease $2pa(G)$ to $a(G)$ in the above upper bound on $a_f(G)$. This raises the following question.

\begin{ques} \label{question}
For a given bounded function $f$, what is the optimal constant $c = c_f$ so that $a_f(G) \le (c + o(1)) \max \{ \Delta_f(G), a(G) \}$ for every multigraph $G$? More generally, what pairs of constants $c, c'$ are optimal so that $a_f(G) \le \max \{ (c + o(1))\Delta_f(G), (c' + o(1))a(G) \}$ for every multigraph $G$?
\end{ques}

Theorem \ref{degree-f-arboricity-approximation} shows that the constant $c_f$ of this question satisfies $c_f \le 2$ for every function $f$. (This could have been derived from the simpler inequality $a_f(G) \le 2pa_f(G)$.) Theorem \ref{degree-f-arboricity-large} shows that generally $c_f > 1$, more precisely $c_f \ge (4t + 7)/(4t + 6)$ whenever $f$ has maximum value $t$, with slight improvements for $t \in \{2,3,4\}$.

We conjecture that Truszczy\'nski's Conjecture \ref{degree-f-arboricity-1} holds when we restrict $G$ to being a simple graph. Our next main results are support for this conjecture. Notice that if $G$ is a simple graph that is $\Delta(G)$-regular, then $a(G) = \lceil (\Delta(G) + 1)/2 \rceil$ by Nash-Williams' Theorem, and then Conjecture \ref{degree-f-arboricity-1} for simple graphs would be implied by the Linear Arboricity Conjecture. But for more general simple graphs $G$, a reduction to the Linear Arboricity Conjecture is not clear. One of our results is that Conjecture \ref{degree-f-arboricity-1} nearly holds for simple graphs with sufficiently large girth. 

\begin{restatable}{thm}{largegirth}\label{large-girth}
Let $G$ be a simple graph, let $f : V(G) \rightarrow \mathbb{Z}_{\ge 2}$ be a function, and let $d = \max \{\Delta_f(G), a(G)\}$. If $G$ has girth $g \ge 4d$, then $a_f(G) \le d + 1$.
\end{restatable}

Our other result is that Conjecture \ref{degree-f-arboricity-1} holds for all simple graphs asymptotically when the function $f$ is bounded.

\begin{restatable}{thm}{asymptoticundirected}\label{asymptotic-undirected}
For every integer $t \ge 2$, there exists a real constant $c_t > 0$ such that for every simple graph $G$ and function $f : V(G) \rightarrow \mathbb{Z}_{\ge 2}$ with maximum value at most $t$, we have
\begin{align*}
	a_f(G) \le d + c_td^{3/4} (\log d)^{1/2},
\end{align*} 
where $d = \max \{\Delta_f(G), a(G) \}$.
\end{restatable}

In particular, the answer to Question \ref{question} is $c_f = 1$ when we restrict to the class of simple graphs. This contrasts with the proof of Theorem \ref{degree-f-arboricity-large}, where the constructed multigraphs have arbitrarily many parallel edges.

The proofs of Theorem \ref{large-girth} and Theorem \ref{asymptotic-undirected} will be extensions of Alon's \cite{Al0} probabilistic proofs for the case $f = 2$ of linear arboricity. These proofs are also found in the book of Alon and Spencer \cite{AlSp}. As with Alon's proofs, it is more convenient to prove our results in the setting of directed graphs $D$. Thus, before our probabilistic proofs we will consider a directed version of a degree-$f$ forest that we call a \textit{degree-$f$ branching}, and we will write a conjecture (Conjecture \ref{degree-f-branchings}) analogous to Conjecture \ref{degree-f-arboricity-1} on what we call the \textit{directed degree-$f$ arboricity} $\vec{a}_f(D)$. We will first prove our large-girth and asymptotic results for directed graphs, and from those we will deduce Theorem \ref{large-girth} and Theorem \ref{asymptotic-undirected}.

We organize the proofs in this paper as follows. In Section \ref{proof-of-main-theorem}, we will prove Theorem \ref{degree-f-arboricity-large}. In Section \ref{directed-version}, we will formulate a version of Conjecture \ref{degree-f-arboricity-1} for directed multigraphs and then show how this directed version implies the undirected version. Finally, in Section \ref{asymptotics} we will prove Theorem \ref{large-girth} and Theorem \ref{asymptotic-undirected}.

\section{Proof of Theorem \ref{degree-f-arboricity-large}} \label{proof-of-main-theorem}
In this section, we will prove Theorem \ref{degree-f-arboricity-large} and thus disprove Conjecture \ref{degree-f-arboricity-1}. We will use a fractional relaxation of degree-$f$ arboricity $a_f(G)$. Let $G$ be a multigraph and let $f : V(G) \rightarrow \mathbb{Z}_{\ge 2}$ be a function. Let $\mathcal{F}$ be the collection of edge-sets of degree-$f$ forests in $G$. We have a variable $y_F$ for every edge-set $F \in \mathcal{F}$. The \textit{fractional degree-$f$ arboricity} $a_f^\ast(G)$ of a multigraph $G$ is the optimal value of the following linear program.
\begin{align*}\arraycolsep=5pt\def\arraystretch{2}
\begin{array}{lrr@{}ll}
	&\text{min}  & \displaystyle\sum_{F \in \mathcal{F}} y_F & &\\
	(P_0) &\text{s.t.} & \displaystyle\sum_{F \in \mathcal{F} \colon e \in F} & y_F \ge 1 & \forall e \in E(G), \\
	&& & y_F \geq 0 & \forall F \in \mathcal{F}.
\end{array}
\end{align*}
Notice that if we add the integrality constraints $y_F \in \{0, 1\}$ for all $F \in \mathcal{F}$, then a feasible solution of $(P_0)$ corresponds to a collection of degree-$f$ forests that covers all of $E(G)$. This implies that $a_f(G) \ge a_f^\ast(G)$. 

We will work with a slightly simplified linear program. Let $G'$ be the underlying simple graph of $G$, and let $\mathcal{F}'$ be the collection of edge-sets of degree-$f$ forests of $G'$. For every edge $e = uv \in E(G')$, let $\mu_e = \mu_G(u,v)$ denote the number of parallel edges between vertices $u$ and $v$ in $G$. It is easy to show that $(P_0)$ has the same optimal value as the following linear program.
\begin{align*}\arraycolsep=5pt\def\arraystretch{2}
\begin{array}{lrr@{}ll}
	&\text{min}  & \displaystyle\sum_{F \in \mathcal{F}'} y_F & &\\
	(P) &\text{s.t.} & \displaystyle\sum_{F \in \mathcal{F}' \colon e \in F} & y_F \ge \mu_e & \forall e \in E(G'), \\
	&& & y_F \geq 0 & \forall F \in \mathcal{F}'.
\end{array}
\end{align*}
The dual of $(P)$ is the following, where $x_e$ is the dual variable associated with edge $e \in E(G')$.
\begin{align*}\arraycolsep=5pt\def\arraystretch{2}
\begin{array}{lrr@{}ll}
	&\text{max}  & \displaystyle\sum_{e \in E(G')} & \mu_e x_e  &\\
	(D) &\text{s.t.} & \displaystyle\sum_{e \in F} & x_e \le 1 & \forall F \in \mathcal{F}', \\
	&& & x_e \geq 0 & \forall e \in E(G').
\end{array}
\end{align*}

Feige, Ravi, and Singh \cite{FeRaSi} proved that the fractional linear arboricity $la^\ast(G) = a_2^\ast(G)$ of a $d$-regular simple graph satisfies $la^\ast(G) \le d/2 + O(\sqrt{d})$. This is only slightly better than Lang and Postle's \cite{LaPo} more recent asymptotic upper bound for $la(G)$, but their proof is simpler. Outside of their work, fractional linear arboricity has not been studied.

To prove Theorem \ref{degree-f-arboricity-large}, we will use the following easy observation. For a multigraph $G$ and integer $m \ge 1$, let $mG$ denote the multigraph obtained from $G$ by replacing every edge of $G$ by $m$ parallel copies of that edge.

\begin{obs} \label{observation}
For every multigraph $G$, function $f : V(G) \rightarrow \mathbb{Z}_{\ge 2}$, and integer $m \ge 1$, we have
\begin{align*}
	a_f(mG) \ge a_f^\ast(mG) = m \cdot a_f^\ast(G).
\end{align*}
\end{obs}

We also observe that $\Delta_f(mG) \le m \cdot \Delta_f(G)$ and $a(mG) \le m \cdot a(G)$. If we can find a multigraph $G$ where $a_f^\ast(G) = c_f \cdot \max \{\Delta_f(G), a(G) \} $ and $c_f > 1$, then
\begin{align*}
	a_f(mG) \ge m \cdot a_f^\ast(G) \ge m \cdot c_f \cdot \max \{\Delta_f(G), a(G) \} \ge c_f \cdot \max \{\Delta_f(mG), a(mG) \},
\end{align*}
which is the required inequality for Theorem \ref{degree-f-arboricity-large}. Therefore our goal is to find, for every constant function $f = t \ge 2$, a multigraph $G = G_t$ and constant $c_t > 1$ such that
\begin{align} \label{goal-equality}
	a_t^\ast(G) \ge c_t \cdot \max \{\Delta_t(G), a(G) \}.
\end{align}

An example of such a multigraph $G_t$ is shown in Figure \ref{counterexample}(a). It consists of a 6-cycle that alternates in one and two parallel edges, an edge connecting two antipodal vertices $u$ and $v$ of the cycle, and $t - 2$ parallel pairs to new vertices attached to each of $u$ and $v$. The underlying simple graph $G_t'$ of $G_t$ is shown in Figure \ref{counterexample}(b), with some edges labeled $e_1, \ldots, e_7$. Consider the following solution $x$ for $(D)$ with respect to $G_t'$:
\begin{align*}
	x_{e} = \begin{cases}
	2/(2t + 3) &\text{if } e \in \{e_1,e_3,e_7\},\\
	1/(2t + 3) &\text{otherwise}.
	\end{cases}
\end{align*}
We will show that $x$ is feasible, using the following claim.

\begin{figure}
	\centering
	\includegraphics[scale=0.94]{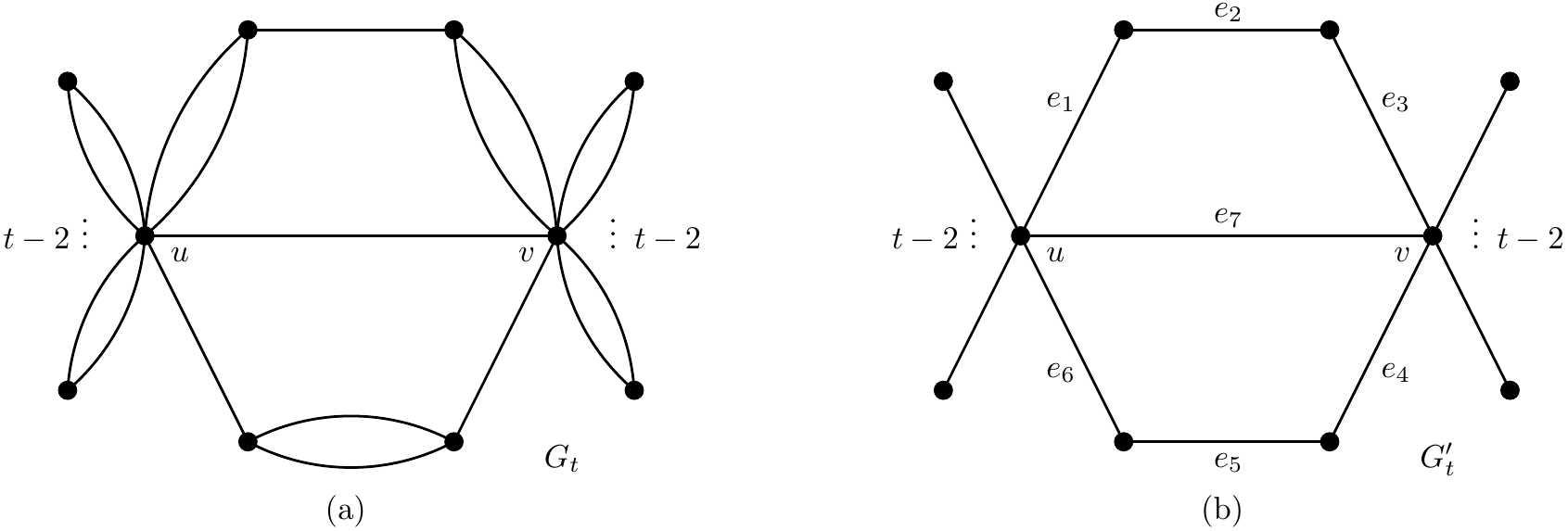}
	\caption{}
	\label{counterexample}
\end{figure}

\begin{claim} \label{claim}
In the simple graph $G_t'$, every degree-$t$ forest has at most $2t + 1$ edges, and every degree-$t$ forest containing each of $e_1, e_3, e_7$ at most $2t$ edges.
\end{claim}

\begin{proof}
The simple graph $G_t'$ has $6 + 2(t - 2) = 2t + 2$ vertices, so every degree-$t$ forest in $G_t'$ has at most $2t + 1$ edges. Suppose that degree-$t$ forest $F$ of $G_t'$ contains $e_1, e_3, e_7$. Then $F$ does not contain $e_2$, but it may contain $e_5$. Besides $e_1$ and $e_7$, $F$ contains at most $t - 2$ additional edges incident to $u$; and besides $e_3$ and $e_7$, $F$ contains at most $t - 2$ additional edges incident to $v$. This implies that $F$ has at most $3 + 1 + 2(t - 2) = 2t$ edges.
\end{proof}

Fix a degree-$t$ forest edge-set $F$ in $G_t'$. If $F$ contains at most two of $e_1, e_3, e_7$, then  by Claim \ref{claim} we have $|F| \le 2t + 1$, so that
\begin{align*}
	\sum_{e \in F} x_e \le 2 \cdot \frac{2}{2t + 3} + (|F| - 2) \cdot \frac{1}{2t + 3} \le 1.
\end{align*}
If $F$ contains each of $e_1, e_3, e_7$, then by Claim \ref{claim} we have $|F| \le 2t$, so that
\begin{align*}
	\sum_{e \in F} x_e = 3 \cdot \frac{2}{2t + 3} + (|F| - 3) \cdot \frac{1}{2t + 3} \le 1.
\end{align*}
Since also $x \ge 0$, this proves that $x$ is feasible for $(D)$.

We calculate objective value of $x$ in $(D)$ to be $(4t + 7)/(2t + 3)$. Thus, by weak duality and the equivalence of $(P)$ and $(P_0)$, we have that $a_t^\ast(G_t) \ge (4t + 7)/(2t + 3)$. (One could prove that $x$ is in fact an optimal solution for $(D)$, by exhibiting a feasible solution $y$ for $(P)$ with the same objective value, but this is not necessary for our proof.) On the other hand, we easily observe that $\Delta_t(G_t) = a(G_t) = 2$. Therefore,
\begin{align*}
	a_t^\ast(G_t) \ge \frac{4t + 7}{2t + 3} = \frac{4t + 7}{4t + 6} \max\{\Delta_t(G_t), a(G_t) \},
\end{align*}
which achieves inequality (\ref{goal-equality}) above and thus proves Theorem \ref{degree-f-arboricity-large} with $c_t \ge \frac{4t + 7}{4t + 6}$. (Note that $\Delta_t(mG_t) = a(mG_t) = 2m$ for every $m \ge 1$, so Theorem \ref{degree-f-arboricity-large} applies to all even $d \ge 2$.)

\begin{figure}
	\centering
	\includegraphics[scale=0.8]{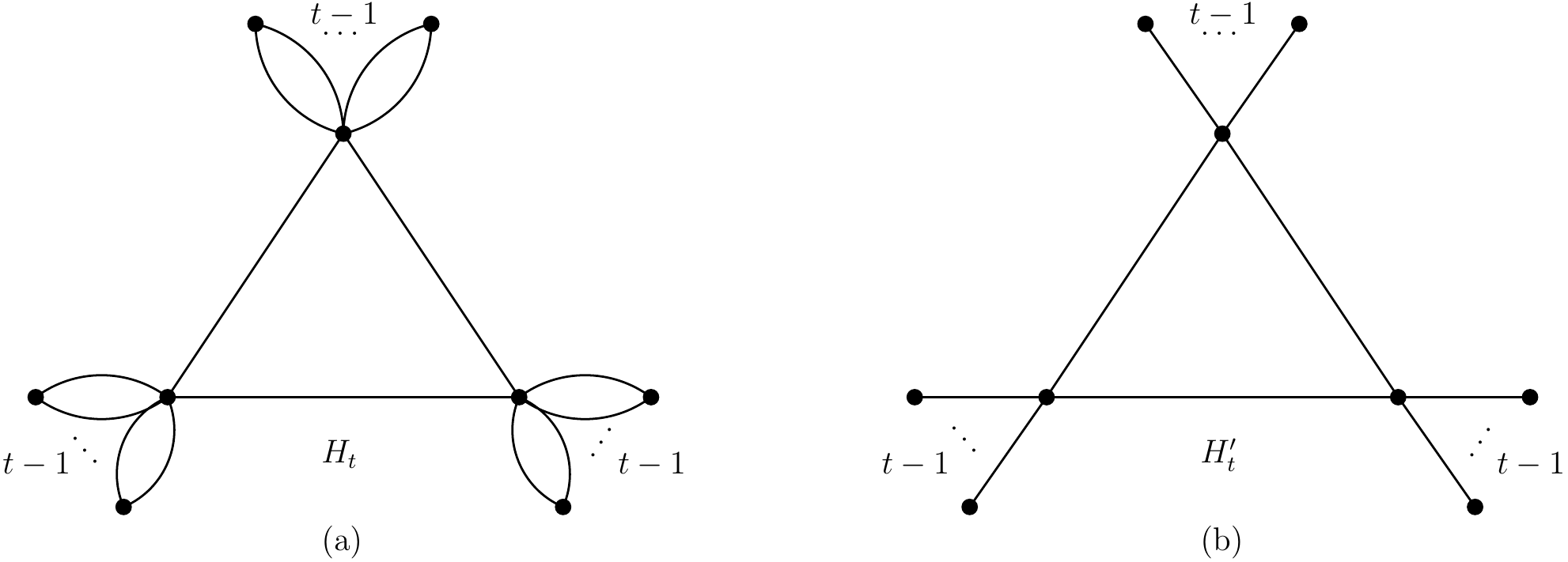}
	\caption{}
	\label{counterexample2}
\end{figure}

For $t \le 4$ we can improve this bound on $c_t$ by replacing $G_t$ with the multigraph $H_t$ shown in Figure \ref{counterexample2}(a). The underlying simple graph $H_t'$ of $H_t$ is shown in Figure \ref{counterexample2}(b). One can show that a feasible solution for $(D)$ with respect to $H_t'$ is $x_e = 1/(3t - 2)$ for every edge $e$, and that this $x$ has objective value $(6t - 3)/(3t - 2)$. Thus $a_t^\ast(H_t) \ge (6t - 3)/(3t - 2)$ while $\Delta_t(H_t) = a(H_t) = 2$, and this gives the bound $c_t \ge (6t - 3)/(6t - 4)$. This bound improves the one above for $t \le 4$, giving $c_2 \ge 9/8$, $c_3 \ge 15/14$, and $c_4 \ge 21/20$.

The idea for constructing the multigraphs $G_t$ and $H_t$ above is to start with a suitable base graph and then to add certain gadgets to a subset of the vertices. These gadgets have the form of a number of parallel edges connecting to a new vertex. If the base graph has large fractional linear arboricity, then by adding a certain number of gadgets we will create a graph with large fractional degree-$t$ arboricity, for any given $t \ge 2$. One can use this approach of adding gadgets to show that it is NP-complete to decide whether a given multigraph $G$ has degree-$f$ arboricity $a_f(G) = 2$, for any fixed function $f : V(G) \rightarrow \mathbb{Z}_{\ge 2}$, starting with the base case $f = 2$ due to P\'eroche \cite{Pe}. Details are found in \cite{Wd1}. 

We observed the inequalities $a_f(G) \ge a_f^\ast(G) \ge \max \{\Delta_f(G), a(G) \}$ and showed that $a_f^\ast(G)$ and $\max \{\Delta_f(G), a(G) \}$ can be arbitrarily far apart. We leave open the question of whether $a_f(G)$ and $a_f^\ast(G)$ can be arbitrarily far apart.

\section{A directed version} \label{directed-version}
We have proven that Conjecture \ref{degree-f-arboricity-1} is false for general multigraphs, but we believe the conjecture holds when we restrict to simple graphs $G$. In Section \ref{asymptotics}, we will prove Theorem \ref{large-girth} and Theorem \ref{asymptotic-undirected} supporting this conjecture. To make it easier to prove these theorems, in this section we introduce an analogue of degree-$f$ arboricity for directed multigraphs and formulate a directed version of Conjecture \ref{degree-f-arboricity-1}. This reformulated conjecture (Conjecture \ref{degree-f-branchings}) is a generalization of the Directed Linear Arboricity Conjecture (Conjecture \ref{directed-linear-arboricity}) due to Nakayama and P\'eroche \cite{NaPe}. 

For a directed multigraph $D$, let $\Delta^-(D)$ denote the maximum indegree $d^-(v)$ among vertices $v$ of $D$, let $\Delta^+(D)$ denote the maximum outdegree, and let $\overline{D}$ denote the underlying undirected multigraph. Two arcs are said to be \textit{parallel} if they have the same head and tail, and \textit{anti-parallel} if the head of one is the tail of the other and vice versa. The \textit{arc-multiplicity} $\mu(D)$ is the maximum number of parallel arcs in $D$. We will call $D$ a \textit{directed graph} if $\mu(D) = 1$, that is, if $D$ has no parallel arcs (but it may have anti-parallel arcs).

A directed graph $B$ is a \textit{branching} if every vertex $v$ of $B$ has indegree $d_B^-(v) \le 1$ and its underlying undirected graph $\overline{B}$ is a forest. An \textit{arborescence} is a branching whose underyling undirected graph is a tree. For a directed multigraph $D$, the \textit{directed arboricity} $\vec{a}(D)$ of $D$ is the minimum number of colors required to color the arcs of $D$ so that every color class is a branching. Observe that $\vec{a}(D) \ge \Delta^-(D)$ and $\vec{a}(D) \ge a(\overline{D})$. Using Edmonds' celebrated theorem on packing arborescences \cite{Ed1, Lo}, Frank \cite{Fr} proved the following formula for directed arboricity $\vec{a}(D)$.

\begin{thm} \label{branchings}
For every directed multigraph $D$, we have
\begin{align*}
	\vec{a}(D) = \max \left\{\Delta^-(D), a(\overline{D}) \right\}.
\end{align*}
\end{thm}

Frank also showed via Nash-Williams' Theorem \cite{Na} that every directed multigraph $D$ satisfies $a(\overline{D}) \le \Delta^-(D) + \mu(D)$, so Theorem \ref{branchings} implies that $\vec{a}(D) \le \Delta^-(D) + \mu(D)$. In particular, if $D$ is a directed graph then
\begin{align*}
	\Delta^-(D) \le \vec{a}(D) \le \Delta^-(D) + 1,
\end{align*}
a result also noted by Kareyan \cite{Ka}.

Theorem \ref{branchings} has some resemblance to Conjecture \ref{degree-f-arboricity-1}. We introduce an $f$-coloring version of branchings to connect the statements rigorously. Let $D$ be a directed multigraph, and let $f : V(D) \rightarrow \mathbb{Z}_{\ge 2}$ be a function. A directed subgraph $B$ of $D$ is a \textit{degree-$f$ branching} if it is a branching where every vertex $v$ has outdegree $d_B^+(v) \le f(v) - 1$. Notice that the underlying undirected graph $\overline{B}$ of a degree-$f$ branching $B$ is a degree-$f$ forest. The \textit{directed degree-$f$ arboricity} $\vec{a}_f(D)$ of directed multigraph $D$ is the minimum number of colors required to color the arcs of $D$ so that every color class is a degree-$f$ branching. 

For the case $f = 2$, a degree-$2$ branching is also called as a \textit{directed linear forest}, and the directed degree-$2$ arboricity $\vec{a}_2(D) = \vec{la}(D)$ is called the \textit{directed linear arboricity} of $D$. Noting that $\vec{la}(D) \ge \Delta^-(D)$ and $\vec{la}(D) \ge \Delta^+(D)$ for every directed multigraph $D$, Nakayama and P\'eroche \cite{NaPe} formulated the \textit{Directed Linear Arboricity Conjecture} for directed graphs $D$:
\begin{align*}
	\vec{la}(D) \le \max \left\{ \Delta^-(D), \Delta^+(D) \right\} + 1.
\end{align*}
Nakayama and P\'eroche proved that this conjecture holds if $D$ is acyclic, if $\Delta^-(D), \Delta^+(D) \le 2$ and $|V(D)| \ge 4$, and if $D$ is one of certain symmetric directed graphs $G^\ast$. (Here, $G^\ast$ is obtained from the undirected graph $G$ by replacing each edge by a pair of anti-parallel arcs.) However, He, Li, Bai, and Sun \cite{HeLiBaSu} later showed that their conjecture does not hold for the complete symmetric directed graphs $K_3^\ast$ and $K_5^\ast$: $\Delta^-(K_n^\ast) = \Delta^+(K_n^\ast) = n - 1$ for all $n$ but $\vec{la}(K_n^\ast) = n + 1$ for $n \in \{3, 5\}$. Still, they believe that these two directed graphs are the only counterexamples, leading to the following updated version of Nakayama and P\'eroche's conjecture.

\begin{conj}[Directed Linear Arboricity Conjecture] \label{directed-linear-arboricity}
	For every directed graph $D$, we have
	\begin{align*}
	\vec{la}(D) \le \max \left\{ \Delta^-(D), \Delta^+(D) \right\} + 1,
	\end{align*}
	except for $D = K_3^\ast, K_5^\ast$, in which case $\vec{la}(D) = \max \left\{ \Delta^-(D), \Delta^+(D) \right\} + 2$.
\end{conj}

Now we consider a generalization of Conjecture \ref{directed-linear-arboricity} for general vertex weight functions $f$ and multigraphs $D$. As before, $\vec{a}_f(D) \ge \Delta^-(D)$ and $\vec{a}_f(D) \ge a(\overline{D})$. Now we also see that $\vec{a}_f(D) \ge \Delta^+_{f-1}(D) = \max_{v \in V(D)} \left\lceil \frac{d^+(v)}{f(v) - 1} \right\rceil$. Based on previously written theorems and conjectures (ignoring Theorem \ref{degree-f-arboricity-large}), it is natural to conjecture the following. 

\begin{conj} \label{degree-f-branchings}
For every directed multigraph $D$, we have
\begin{align*}
	\vec{a}_f(D) \le \max \left\{\Delta^-(D), \Delta_{f-1}^+(D), a(\overline{D}) \right\} + 1.
\end{align*}
\end{conj}

We will show that Conjecture \ref{degree-f-branchings} nearly implies Conjecture \ref{degree-f-arboricity-1}. Because Conjecture \ref{degree-f-arboricity-1} is false for general multigraphs as we have shown, so is Conjecture \ref{degree-f-branchings}. However, it could still be true for directed graphs (with no parallel arcs). Since a directed graph $D$ satisfies $\Delta^-(D) \le a(\overline{D}) \le \Delta^-(D) + 1$ as noted above, Conjecture \ref{degree-f-branchings} for directed graphs is basically the statement that $D$ satisfies
\begin{align*}
	\vec{a}_f(D) \le \max \left\{ \Delta^-(D)+1, \Delta_{f-1}^+(D) \right\} + 1.
\end{align*}

In the case $f = 2$, it is an easy observation that the Directed Linear Arboricity Conjecture (Conjecture \ref{directed-linear-arboricity}) nearly implies the Linear Arboricity Conjecture (Conjecture \ref{linear-arboricity}): Given an undirected simple graph $G$, let $D$ be a balanced orientation of $G$, meaning that every vertex $v$ has both indegree and outdegree at most $\lceil d_G(v)/2 \rceil$ in $D$. (The existence of balanced orientations is an easy exercise from the theory of Euler tours or network flows, and it is also a special case of Theorem \ref{orientation-theorem} below.) Then $\Delta^-(D), \Delta^+(D) \le \lceil \Delta(G)/2 \rceil$, and so by Conjecture \ref{directed-linear-arboricity} (for $D \neq K_3^\ast, K_5^\ast$) we have $la(G) \le \vec{la}(D) \le \max \{\Delta^-(D), \Delta^+(D)\} + 1 \le \lceil \Delta(G)/2 \rceil + 1$, which is almost the Linear Arboricity Conjecture. To make this kind of reduction work for more general functions $f$, we will replace balanced orientations with the following more general orientation theorem due to Entringer and Tolman \cite{EnTo} (see also \cite{Wd}).

\begin{thm}[Entringer-Tolman] \label{orientation-theorem}
Given a multigraph $G$ (possibly with loops) and functions $g, h : V(G) \rightarrow \mathbb{Z}_{\ge 0}$, $G$ can be oriented so that every vertex $v \in V(G)$ has indegree $d^-(v) \le g(v)$ and outdegree $d^+(v) \le h(v)$ if and only if
\begin{enumerate}[label={(\arabic*)}]
	\item $d(v) \le g(v) + h(v)$ for all $v \in V(G)$, and
	\item $e(S) \le \min \{ g(S), h(S) \}$ for all $S \subseteq V(G)$.
\end{enumerate}
\end{thm}

\begin{cor} \label{orientation-corollary}
For every multigraph $G$ and function $f : V(G) \rightarrow \mathbb{Z}_{\ge 2}$, if $\max \{\Delta_f(G), a(G) \} \le d$ then $G$ has an orientation $D$ such that $\max \{\Delta^-(D), \Delta_{f-1}^+(D), a(\overline{D}) \} \le d$.
\end{cor}

\begin{proof}
Define $g, h : V(G) \rightarrow \mathbb{Z}_{\ge 1}$ by $g(v) = d$ and $h(v) = d(f(v) - 1)$ for all $v \in V(G)$. Then $d_G(v) \le d \cdot f(v) \le g(v) + h(v)$ for all $v \in V(G)$, and $e(S) \le d(|S| - 1) < d|S| = \min \{g(S), h(S)\}$ for all $S \subseteq V(G)$. By Theorem \ref{orientation-theorem}, $G$ has an orientation $D$ such that every vertex $v$ has indegree at most $d$ and outdegree at most $d(f(v) - 1)$. That is, $\Delta^-(D) \le d$ and $\Delta_{f-1}^+(D) \le d$. We also see that $a(\overline{D}) = a(G) \le d$.
\end{proof}

Therefore, for any multigraph $G$, if $d = \max \{\Delta_f(G), a(G) \}$ and $D$ is from Corollary \ref{orientation-corollary}, then Conjecture \ref{degree-f-branchings} implies that
\begin{align*}
	a_f(G) \le \vec{a}_f(D) \le \max \{\Delta^-(D), \Delta_{f-1}^+(D), a(\overline{D}) \} + 1 \le d + 1 = \max \{\Delta_f(G), a(G) \} + 1,
\end{align*}
which is almost Conjecture \ref{degree-f-arboricity-1} as what we wanted to show. This reduction to directed graphs will be used in our proofs of Theorem \ref{large-girth} and Theorem \ref{asymptotic-undirected}.

We remark that Conjecture \ref{degree-f-branchings} can be viewed as a matroid problem, about how closely the covering number $\beta(M_1, M_2, M_3)$ of three matroids $M_1, M_2, M_3$ on ground set $E(D)$ is determined by the covering number $\beta(M_1), \beta(M_2), \beta(M_3)$ of each of these matroids individually. We will not comment further on this perspective.


\section{Large girth and asymptotics} \label{asymptotics}
In this section, we will prove Theorem \ref{large-girth} and Theorem \ref{asymptotic-undirected}. Specifically, we will first prove that Conjecture \ref{degree-f-branchings} on degree-$f$ branchings holds for directed graphs with large directed girth, and then prove that it holds for all directed graphs asymptotically when the function $f$ is bounded. Applying Corollary \ref{orientation-corollary}, these results imply the desired theorems. Recall that a directed graph $D$ is taken to have no parallel arcs, but it may have anti-parallel arcs.

Our proofs are extensions of the probabilistic proofs of Alon \cite{Al0}, who proved such partial results for the Directed Linear Arboricity Conjecture (the case $f = 2$) while improving and simplifying his original arguments in \cite{Al1}. (See also Alon and Spencer \cite{AlSp}.) Alon reduced to and wrote his proofs specifically for $d$-regular (directed) graphs, but we cannot do the same reduction for $f \neq 2$, so we write our proofs in the general non-regular setting. Recall that Conjecture \ref{degree-f-branchings} for a directed graph $D$ is close to, but not exactly, the inequality
\begin{align*}
	\vec{a}_f(D) \le \max \left\{ \Delta^-(D) + 1, \Delta_{f-1}^+(D) \right\} + 1.
\end{align*} 

\subsection{Large girth}
First we prove that Conjecture \ref{degree-f-branchings} holds for directed graphs with large directed girth. Here, the \textit{directed girth} is the length of a shortest directed cycle. We will need Hakimi and Kariv's \cite{HaKa} $f$-coloring generalization of K\"onig's edge-coloring theorem, stated as follows.

\begin{thm} \label{Konig-f}
For every bipartite multigraph $G$ and function $f : V(G) \rightarrow \mathbb{Z}_{\ge 1}$, we have $\chi_f'(G) = \Delta_f(G)$.
\end{thm}

We will also need a result on independent transversals in graphs. Given a simple graph $G$ and a partition $(V_i)_{i \in [k]}$ of its vertex set $V(G)$, an \textit{independent transversal} of $G$ with respect to $(V_i)_{i \in [k]}$ is an independent set of $G$ containing one vertex from each vertex class $V_i$. Aharoni, Alon, and Berger \cite{AhAlBe} proved the following result on independent transversals in line graphs.

\begin{thm} \label{independent-transversals}
Let $G$ be the line graph of a simple graph and let $(V_i)_{i \in [k]}$ be a partition of its vertex set $V(G)$. If $|V_i| \ge \Delta(G)+2$ for every $i$, then $G$ has an independent transversal with respect to $(V_i)_{i \in [k]}$.
\end{thm}

We now prove our result on directed graphs with large directed girth. The main differences in our proof compared to that of Alon \cite{Al0} for the case $f = 2$ are that we use a different initial arc-coloring of the directed graph, and that we apply Theorem \ref{independent-transversals} to only an induced subgraph of our line graph rather than the entire line graph.

\begin{thm} \label{large-directed-girth}
Let $D$ be a directed graph, let $f : V(D) \rightarrow \mathbb{Z}_{\ge 2}$ be a function, and let $d = \max \{\Delta^-(D), \Delta_{f-1}^+(D) \}$. If $D$ has directed girth $g \ge 4d$, then $\vec{a}_f(D) \le d + 1$.
\end{thm}

\begin{proof}
Construct an auxiliary bipartite graph $G$ with parts $X$ and $Y$ as follows: for every vertex $v$ of $D$ put a copy $v_X$ in $X$ and $v_Y$ in $Y$, and for every arc $e$ of $D$ with tail $u$ and head $v$ put an edge in $G$ between $u_X$ and $v_Y$. Then $v_X$ has degree $d_D^+(v)$ and $v_Y$ has degree $d_D^-(v)$ in $G$, for every $v \in V(D)$. Define the function $g : V(G) \rightarrow \mathbb{Z}_{\ge 1}$ by $g(v_X) = f(v) - 1$ and $g(v_Y) = 1$ for all $v \in V(D)$. By Theorem \ref{Konig-f}, 
\begin{align*}
	\chi_g'(G) = \Delta_g(G) = \max \left\{ \max_{v \in V(D)} \left\lceil \frac{d_D^+(v)}{f(v) - 1} \right\rceil, \max_{v \in V(D)} d_D^-(v) \right\} = \max \left\{ \Delta_{f-1}^+(D),  \Delta^-(D) \right\} = d.
\end{align*}
Thus $G$ can be edge-colored into $d$ degree-$g$ subgraphs. These subgraphs of $G$ correspond to directed subgraphs $B_1, \ldots, B_d$ of $D$ with $d_{B_i}^-(v) \le 1$ and $d_{B_i}^+(v) \le f(v) - 1$ for all $1 \le i \le d$ and $v \in V(D)$. The $B_i$'s are thus directed degree-$f$ pseudoforests, meaning that every vertex $v$ of $B_i$ has indegree at most one and the underlying undirected graph $\overline{B_i}$ is a degree-$f$ pseudoforest. 

Observe that if we remove one arc from every monochromatic directed cycle in $D$, the remaining color classes will be degree-$f$ branchings. Let $D'$ be the spanning directed subgraph of $D$ that is the union all monochromatic directed cycles in $D$, let $H$ be the line graph of $\overline{D'}$, and let $(V_i)_{i \in [k]}$ be the edge sets of the monochromatic directed cycles in $D'$. Then $(V_i)_{i \in [k]}$ is a partition of $V(H)$, and by the directed girth condition we have $|V_i| \ge 4d$ for all $i \in [k]$. Since $\Delta(H) \le 4d - 2$, $|V_i| \ge 4d = (4d - 2) + 2$ for every $i$, and $H$ is the line graph of a simple graph (as $D$ cannot have anti-parallel arcs by the directed girth assumption), Theorem \ref{independent-transversals} implies that there is an independent transversal of $H$ with respect to $(V_i)_{i \in [k]}$. But this means that there is a (directed) matching $M$ of $D'$ containing an arc from every monochromatic directed cycle in $D$. Then $M, B_1 \setminus M, \ldots, B_d \setminus M$ are all degree-$f$ branchings, giving us an arc-coloring of $D$ into $d + 1$ degree-$f$ branchings.
\end{proof}

By Corollary \ref{orientation-corollary}, Theorem \ref{large-directed-girth} implies the desired large-girth result on Conjecture \ref{degree-f-arboricity-1}.

\largegirth*

Alon originally proved these large-girth results for the case $f = 2$ under the girth condition $g \ge 100d$ \cite{Al1}, which he later improved to $g \ge 8ed$ \cite{Al0}. He derived these girth conditions from earlier versions of Theorem \ref{independent-transversals} on independent transversals. Using the Lov\'asz local lemma, Alon proved Theorem \ref{independent-transversals} for general graphs $G$ (not just line graphs) under the conditions $|V_i| \ge 25\Delta(G)$ \cite{Al1} and $|V_i| \ge 2e\Delta(G)$ \cite{Al0}, respectively. Haxell \cite{Hax} subsequently improved these class size conditions to $|V_i| \ge 2\Delta(G)$. Haxell's bound is best possible if one does not assume that $G$ is the line graph of a simple graph \cite{SzTa}.

\subsection{Asymptotics}

Now we prove that Conjecture \ref{degree-f-branchings} holds asymptotically for all directed graphs $D$ when the function $f$ is bounded. Similar to Alon's \cite{Al0} proof for the case $f = 2$, we show that every directed graph $D$ can be decomposed into a specified number of directed subgraphs each with large directed girth and approximately the same maximum degree, and then we apply Theorem \ref{large-directed-girth} to each of these directed subgraphs individually. We use the following slight modification of a lemma of Alon \cite{Al0}.

\begin{lem} \label{random-vertex-coloring}
Let $D$ be a directed graph and let $f : V(D) \rightarrow \mathbb{Z}_{\ge 2}$ be a function. Suppose that $d = \max \{ \Delta^-(D), \Delta_{f - 1}^+(D) \}$ is sufficiently large compared to the maximum value of $f$, and let $k \le d^{9/10}$ be a positive integer. Then there is a $k$-coloring of $V(D)$ using the colors $0, 1, \ldots, k - 1$ with the following property: for every vertex $v$ and color $i$, the numbers
\begin{align*}
	d^-(v, i) &= |\{ u \in V(D) : (u, v) \in E(D) \text{ and } u \text{ is colored } i \}|, \\
	d^+(v, i) &= |\{ u \in V(D) : (v, u) \in E(D) \text{ and } u \text{ is colored } i \}|
\end{align*}
satisfy $d^-(v, i), \frac{d^+(v, i)}{f(v) - 1} \le \frac{d}{k} + 3 \sqrt{\frac{d \log d}{k}}.$

\end{lem}

\begin{proof}
Start by augmenting $D$ to a directed graph $D'$, adding auxiliary vertices and arcs so that every vertex $v \in V(D)$ has indegree $d_{D'}^-(v) = d$ and outdegree $d_{D'}^+(v) = d(f(v) - 1)$ in $D'$ (not caring about the indegrees and outdegrees of the added vertices). Consider a random coloring of the vertices of $D'$ with the colors $0, 1, \ldots, k-1$, where the color of every vertex chosen uniformly at random. For every vertex $v \in V(D)$ and color $i$, let $A_{v,i}^-$ be the event that $d_{D'}^-(v,i) > d/k + 3\sqrt{(d \log d)/k}$. Observe that $d_{D'}^-(v,i)$ is a binomial random variable with mean $d/k$. By a version of Chernoff's inequality (see Appendix A in \cite{AlSp}), we have that $\text{Pr}[A_{v,i}^-] \le 1/d^4$. Likewise, letting $A_{v,i}^+$ be the event that $d_{D'}^+(v,i)/(f(v) - 1) > d/k + 3\sqrt{(d \log d)/k}$, we have that $\text{Pr}[A_{v,i}^+] \le 1/d^{4(f(v) - 1)} \le 1/d^4$.

Each of the events $A_{v,i}^-, A_{v,i}^+$ is mutually independent of the events $A_{u,j}^-, A_{u,j}^+$ for all the vertices $u \in V(D)$ that do not have a common neighbor with $v$ in $\overline{D}$. Thus each of $A_{v,i}^-, A_{v,i}^+$ is mutually independent of all but at most $k(td)^2$ of the events $A_{u,j}^-, A_{u,j}^+$, where $t$ is the maximum output of $f$. Since $e(1/d^4)(k(td)^2 + 1) < 1$ for $d$ sufficiently large compared to $t$, by the symmetric Lov\'asz local lemma (see Chapter 5 in \cite{AlSp}) no event $A_{v,i}^-$ or $A_{v,i}^+$ occurs. Thus there is a coloring of $V(D')$ satisfying $d_{D'}^-(v, i), d_{D'}^+(v, i)/(f(v) - 1) \le d/k + 3 \sqrt{(d \log d)/k}$ for all $v \in V(D)$ and $0 \le i \le k-1$. Deleting the auxiliary vertices and arcs from $D'$, this gives a desired coloring of $V(D)$.
\end{proof}

We can now prove the following asymptotic version of Conjecture \ref{degree-f-branchings} for directed graphs.

\begin{thm} \label{asymptotic-directed}
For every integer $t \ge 2$, there exists a constant $c_t > 0$ such that for every directed graph $D$ and function $f : V(D) \rightarrow \mathbb{Z}_{\ge 2}$ with maximum value at most $t$, we have
\begin{align*}
	\vec{a}_f(D) \le d + c_td^{3/4} (\log d)^{1/2},
\end{align*} 
where $d = \max \{\Delta^-(D), \Delta_{f-1}^+(D) \}$.
\end{thm}

\begin{proof}
We may assume that $d$ is sufficiently large compared to $t$ wherever necessary. Pick a prime $k$ satisfying $5d^{1/2} \le k \le 10d^{1/2}$. By Lemma \ref{random-vertex-coloring}, there exists a $k$-coloring $\phi$ of $V(D)$ satisfying the stated inequalities. For each $0 \le i \le k - 1$, let $D_i$ be the spanning directed subgraph of $D$ with arc set $E(D_i) = \{ (u,v) \in E(D) : \phi(v) \equiv \phi(u) + i \pmod{k} \}$. The inequalities in Lemma \ref{random-vertex-coloring} imply that $d_i = \max \{ \Delta^-(D_i), \Delta_{f-1}^+(D_i) \} \le d/k + 3 \sqrt{(d \log d)/k}$ for every $0 \le i \le k - 1$. Moreover, for $i \neq 0$ the length of every directed cycle in $D_i$ is divisible by $k$, and thus $D_i$ has directed girth $g_i \ge k \ge 4d_i$ (using that $k \ge 5d^{1/2}$ and $d$ is sufficiently large). By Theorem \ref{large-directed-girth}, we deduce that $\vec{a}_f(D_i) \le d_i + 1 \le d/k + 3 \sqrt{(d \log d)/k} + 1$ for every $1 \le i \le k - 1$. For $D_0$ we only use the trivial inequality $\vec{a}_f(D_0) \le 2d_0 \le 2d/k + 6 \sqrt{(d \log d)/k}$, obtained by first using Theorem \ref{Konig-f} to arc-color $D_0$ into $d_0$ directed degree-$f$ pseudoforests, and then trivially arc-coloring each of these directed degree-$f$ pseudoforests into $2$ degree-$f$ branchings. These inequalities together with $5d^{1/2} \le k \le 10d^{1/2}$ give us that
\begin{align*}
	\vec{a}_f(D) \le (k - 1) \left( \frac{d}{k} + 3 \sqrt{\frac{d \log d}{k}} + 1 \right) + \left( \frac{2d}{k} + 6 \sqrt{\frac{d \log d}{k}} \right) \le d + c_td^{3/4}( \log d )^{1/2},
\end{align*}
for some constant $c_t$ depending on $t$ (since we assumed $d$ is large compared to $t$).
\end{proof}

By Corollary \ref{orientation-corollary}, Theorem \ref{asymptotic-directed} implies the desired asymptotic version of Conjecture \ref{degree-f-arboricity-1} for simple graphs.

\asymptoticundirected*

We conclude with a few remarks. Alon \cite{Al0} stated that the lower order term in Theorem \ref{asymptotic-directed} (and thus also in Theorem \ref{asymptotic-undirected}) can be improved to $c'd^{2/3}(\log d)^{1/3}$ when $f = 2$, by using ``recursion" instead of a naive arc-coloring of $D_0$, as well as changing some parameters. This kind of modification could perhaps also work for more general vertex weight functions $f$. Alon, Teague, and Wormald \cite{AlTeWo} later recovered the same lower order term stated by Alon, in the undirected case, using a different and in some ways simpler method. Instead of using a large girth result like Theorem \ref{large-directed-girth} or Theorem \ref{large-girth} above, they use a variant of the classical and easier result that the complete graph $K_{2n}$ can be decomposed into $n$ Hamiltonian paths. However, their proof method does not appear to generalize as readily as Alon's for general functions $f$.

Alon's \cite{Al1} original asymptotic proof of the Linear Arboricity Conjecture involved coloring the edges of the graph rather than the vertices as we did above. This original approach resulted in a worse error term, but Alon noted without proof that this approach also gives the desired asymptotics more generally for multigraphs $G$ with bounded edge-multiplicity $\mu(G)$. One can check that this approach also generalizes well to other vertex weight functions $f$ on the vertices, when one works in the directed setting. That is, the asymptotic upper bound $a_f(G) \le (1 + o(1))\max \{\Delta_f(G), a(G)\}$ that we proved above holds more generally for multigraphs $G$ with bounded edge-multiplicity. In particular, the optimal constant in Question \ref{question} is $c_f = 1$ more generally for such multigraphs.

\section*{Acknowledgments}
I would like to thank Nathan Benedetto Proen\c{c}a for aiding my understanding and exposition of the linear programming in Section \ref{proof-of-main-theorem}. I would also like to thank Penny Haxell for the helpful edits.

\end{document}